\documentclass[a4paper, leqno]{amsart}

\usepackage{amssymb, amsmath, amsthm}
\usepackage[numbers]{natbib}
\usepackage{xcolor}
\numberwithin{equation}{section}
\newtheorem{thm}{Theorem}
\newtheorem{theo}[thm]{Theorem}

\newtheorem{lemma}[thm]{Lemma}
\newtheorem{proposition}[thm]{Proposition}

\newtheorem{rmk}[thm]{Remark}

\newtheorem*{theorem*}{Theorem}

\newtheorem*{definition}{Definition}
 
\newcommand{\dd}{{d}}
\newcommand{\T}{\mathbb{T}}

\newcommand{\brkt}[1]{\left(#1\right)}

\newcommand{\abs}[1]{\left|#1\right|}

\title[Stein and Zygmund theorems for $H^{\log} (\mathbb{R}^d)$]{$L\log \log L$ versions of Stein's and Zygmund's theorems for the Hardy space $H^{\log}(\mathbb{R}^d)$}

\author[Bakas]{Odysseas Bakas}
\address{Centre for Mathematical Sciences, Lund University, 221 00 Lund, Sweden}
\email{odysseas.bakas@math.lu.se}
\author[Rodr\'iguez-L\'opez]{Salvador Rodr\'iguez-L\'opez}
\address{Department of Mathematics, Stockholm University, 106 91 Stockholm, Sweden}
\email{s.rodriguez-lopez@math.su.se}
\author[Sola]{Alan Sola}
\address{Department of Mathematics, Stockholm University, 106 91 Stockholm, Sweden}
\email{sola@math.su.se}
\date{\today}
\thanks{The first author was partially supported by the ``Wallenberg Mathematics Program 2018'', grant no. KAW 2017.0425, financed by the Knut and Alice Wallenberg Foundation.
The second author was partially supported by the Spanish Government grant MTM2016-75196-P}


\begin{document}

\maketitle

\begin{abstract} We obtain versions of some classical results of Zygmund and Stein for functions belonging to the Hardy space $H^{\log} (\mathbb{R}^d)$ introduced by Bonami, Grellier, and Ky. We present further applications in the context of more general Orlicz spaces. This yields slight extensions of results previously obtained by Bonami-Madan, Iwaniec-Verde, and others.
 \end{abstract}

\section{Introduction}
The Hardy-Littlewood maximal function is a fundamental object in harmonic analysis, defined for a locally integrable function $f\colon \mathbb{R}^d\to \mathbb{C}$ by setting
$$ M (f) (x) := \sup_{r>0} \frac{1}{|B (x,r)|} \int_{B(x,r)} |f(y)| dy \quad  x \in \mathbb{R}^d,$$
where $B(x,r)$ denotes the open ball in $\mathbb{R}^d$ centered at $x$ with radius $r>0$ and $|A|$ denotes the Lebesgue measure of $A \subseteq \mathbb{R}^d$. It is a basic fact that the mapping $f\mapsto M(f)$ is bounded on $L^p(\mathbb{R}^d)$ for $1<p\leq \infty$. The maximal operator is also bounded from $L^1(\mathbb{R}^d)$ to weak-$L^1$, but does not map $L^1(\mathbb{R}^d)$ to itself (see, for instance, \cite{Big_Stein} for an in-depth discussion).

However, $M(f)$ is locally integrable provided $f$ is compactly supported and satisfies the $L\log L$ condition
\[\int_{\mathbb{R}^d}|f(x)|\log^+|f(x)|dx<\infty,\]
where, as usual, $\log^+|x|=\max\{\log|x|, 0\}$.
In a 1969 paper, E.M. Stein \cite{Stein_LlogL} proved that this $L\log L$ condition is both sufficient and necessary for integrability of the Hardy-Littlewood maximal function, in the following sense: if $f$ is supported in some finite ball $B=B(r)$ of radius $0<r<\infty$, then 
\[\int_{B}M(f) dx<\infty \quad \textrm{if, and only if,}\quad \int_{B}|f(x)|\log^+|f(x)|dx<\infty.\]

Another classical result that involves the space $L\log L$ is due to Zygmund, and asserts that the periodic Hilbert transform maps $L\log L(\mathbb{T})$ to $L^1(\mathbb{T})$; see e.g. Theorem 2.8 in Chapter VII of \cite{Zygmund_book}. This implies that $L\log L(\mathbb{T})$ is contained in the real Hardy space $H^1(\mathbb{T})$ consisting of integrable functions on the torus whose Hilbert transforms are integrable. Moreover, as shown by Stein in \cite{Stein_LlogL}, Zygmund's theorem has a partial converse, namely if $f\in H^1(\mathbb{T})$ and $f$ is non-negative, then $f$ necessarily belongs to $L\log L(\mathbb{T})$. Therefore, in view of the aforementioned results of Zygmund and Stein, the Hardy  space $H^1(\mathbb{T})$ is,  in  terms  of  magnitude, associated with the Orlicz space $L\log L(\mathbb{T})$. 

In this note, we obtain versions of these results for the Musielak-Orlicz Hardy space $H^{\log}(\mathbb{R}^d)$ that was recently introduced by A. Bonami, S. Grellier, and L.D. Ky in \cite{BGK} and further studied by Ky in \cite{Ky}. See also \cite{B_etal} and \cite{YYK}. To do this, we identify the correct analog of $L\log L$ in this context, which turns out to be $L\log \log L$:  given a measurable subset $B$ of $\mathbb{R}^d$, $L \log \log L (B)$ denotes the class of all locally integrable functions $f$ with $\mathrm{supp}(f) \subseteq B$ and
$$ \int_B |f(x)| \log^+ \log^+ |f(x)| dx < \infty.$$ 

In order to formally state our results, we now give the definition of the space $H^{\log}(\mathbb{R}^d)$. Let $\Psi \colon \mathbb{R}^d \times [0, \infty) \rightarrow [0, \infty)$ denote the function given by
\[\Psi (x,t) : = \frac{t}{ \log (e + t) + \log (e + |x|)}, \quad (x,t) \in \mathbb{R}^d \times [0, \infty ).\] 
If $B$ is a subset of $\mathbb{R}^d$, one defines $L_{\Psi} (B)$ to be the space of all locally integrable functions $f$ on $B$ satisfying 
$$ \int_B \Psi (x, |f(x)|) dx < \infty.$$

We shall also fix a non-negative function $\phi \in C^{\infty} (\mathbb{R}^d)$, which is supported in the unit ball of $ \mathbb{R}^d$ and has $\int_{\mathbb{R}^d} \phi (y) dy = 1$ and $\phi (x) = c_0$ for all $|x| \leq 1/2$, where $c_0$ is a constant. Given an $\epsilon > 0$, we use the standard notation $ \phi_{\epsilon} (x) : = \epsilon^{-d} \phi(\epsilon^{-1} x)$, $x \in \mathbb{R}^d$. 

\begin{definition}[$H^{\log}$, see \cite{BGK, YYK}]
If $\phi$ is as above, consider the maximal function
$$ M_{\phi} (f) (x) := \sup_{\epsilon > 0} | (f \ast \phi_{\epsilon}) (x) |, \quad x \in \mathbb{R}^d.$$

The Hardy space $H^{\log} (\mathbb{R}^d)$ is defined to be the space of tempered distributions $f$ on $\mathbb{R}^d $ such that $ M_{\phi} (f) \in L_{\Psi} (\mathbb{R}^d)$, that is, $M_{\phi}(f)$ satisfies
$$ \int_{\mathbb{R}^d} \Psi (x, M_{\phi} (f) (x)) dx < \infty .$$ 
\end{definition}
The motivation for defining the space $H^{\log}$ comes from the study of products of functions in the real Hardy space $H^1(\mathbb{R}^d)$ and functions in $\mathrm{BMO}(\mathbb{R}^d)$, the class of functions of bounded mean oscillation. Following earlier work by Bonami, T. Iwaniec, P. Jones, and M. Zinsmeister in \cite{BIJZ}, it was shown by Bonami, Grellier, and Ky \cite{BGK} that the product $fg$, in the sense of distributions, of a function 
$f\in H^1(\mathbb{R}^d)$ and a function $g\in \mathrm{BMO}(\mathbb{R}^d)$ can be represented as a sum of a continuous bilinear mapping into $L^1(\mathbb{R}^d)$ and a
continuous bilinear operator into $H^{\log}(\mathbb{R}^d)$.

Here is our version of Stein's lemma for $L_{\Psi}$.  
\begin{thm}\label{Stein-type_lemma}
Let $f $ be a measurable function supported in a closed ball $B \subsetneq \mathbb{R}^d$.

Then $M(f) \in L_{\Psi} (B)$ if, and only if, $f \in L \log \log L (B)$.
\end{thm} 
Our proof in fact leads to a more general version of Theorem \ref{Stein-type_lemma}. We discuss this, and give a proof of Theorem \ref{Stein-type_lemma} in Section \ref{steinproof}.

Next is the analog of Zygmund's result for $H^{\log}(\mathbb{R}^d)$.
\begin{theo}\label{LlogL} Let $B$ denote the closed unit ball in $\mathbb{R}^d$.

If $f$ is a measurable function satisfying $f \in L \log \log L (B)$ and $\int_{B} f (y) dy = 0$, then $f \in H^{\log}(\mathbb{R}^d)$.
\end{theo}
We remark that the mean-zero condition in the hypothesis is in fact necessary in order to place a compactly supported function in $H^{\log}$.
The proof of Theorem \ref{LlogL} is presented in Section 3. 

In Section 4, we discuss further extensions to the periodic setting.

\begin{rmk}
After posting a first version of this note, the authors were informed that our main results can be derived from results previously obtained  in the setting of Orlicz spaces; see for instance \cite{BM,IV}. We are grateful for having been directed to the appropriate sources. In this note, we give a self-contained account, including a discussion of sharpness, and indicate some minor modifications that need to be made to obtain results in the Musielak-Orlicz setting. 
\end{rmk}

\section{Proof of the Stein-type Theorem for $L_{\Psi}$ and further extensions}\label{steinproof}

We begin with an elementary observation that will be implicitly used several times in the sequel: if $\Phi : [0, \infty) \rightarrow [0, \infty)$ is an increasing function, then for every positive constant $\alpha_0$ one has 
$$ \int_B \Phi (|g (x)| ) dx \leq \Phi (\alpha_0) |B| + \int_{\{ |g| > \alpha_0 \}} \Phi ( |g(x)| ) dx $$ 
for each measurable set $B$ in $\mathbb{R}^d$ with finite measure.

We now turn to the proof of our first theorem. 
\begin{proof}[Proof of Theorem \ref{Stein-type_lemma}]
Assume first that $f \in L \log \log L (B)$. The main observation is that locally the space $L_{\Psi}$ essentially coincides with the Orlicz space defined in terms of the function $\Psi_0 (t) : = t \cdot [\log (e + t)]^{-1}$, $t \geq 0$ and so, one can employ the arguments of Stein \cite{Stein_LlogL}. In view of this observation, we remark that the fact that $f \in L \log \log L (B)$ implies $M (f ) \in L_{\Psi_0} (B)$ is well-known; see for instance \cite[p.242]{BM}, \cite[Sections 4 and 7]{IV}. We shall also include the proof of this implication here for the convenience of the reader.

To be more precise, we note that for $x\in B$ one has
\begin{equation}\label{eq:equivalence}
    \log(e+M(f)(x))\leq \log \brkt{(e+|x|)(e+M(f)(x))}  
\leq c\log(e+M(f)(x)),
\end{equation} 
for a constant $c$ that only depends on $B$.  
Next, an integration by parts yields
\[\int_{e}^{y} \frac{1}{\log \alpha}d\alpha=\frac{y}{\log y}-e+\int_{e}^{y}\frac{1}{\log^2 \alpha}d\alpha,\]
so that
\[\frac{y}{\log y}\leq e +\int_{e}^{y}\frac{1}{\log \alpha}d\alpha, \quad \textrm{for} \quad y>e.\]
Together, these two observations imply that
\begin{align*}
 \int_B \Psi (x, M(f)(x) ) dx &\lesssim_B 1 + \int_{ B \cap \{ M(f) > e \} } \Bigg( \int_e^{M(f)(x)} \frac{1}{ \log \alpha } d \alpha \Bigg) dx \\
 &=1 +  \int_e^{\infty} \frac{1}{ \log \alpha } \cdot | \{ x \in B : M (f)(x) > \alpha \} | d \alpha.
\end{align*}
To estimate the last integral, note that there exists an absolute constant $C_d >0$ such that
\begin{equation}\label{wt_sv}
|\{ x \in \mathbb{R}^d : M (f)(x) > \alpha \}| \leq \frac{C_d}{\alpha} \int_{\{ |f| > \alpha/2 \}} |f(x)| dx  
\end{equation}
 for all $\alpha > 0$; see e.g. \cite[(5)]{Stein_LlogL} or Section 5.2 (a) in Chapter I in \cite{Singular}. We thus deduce from \eqref{wt_sv} that
\begin{align*}
\int_B \Psi (x, M(f)(x)) dx &\lesssim_B 1 +  \int_B |f(x)| \cdot \Bigg( \int_e^{2|f(x)|} \frac{1}{\alpha \log \alpha} d \alpha \Bigg) dx \\
&\lesssim 1 + \int_{B} |f(x)| \log^+ \log^+ |f(x)| dx,  
 \end{align*}
which implies that $ M (f) \in L_{\Psi} (B)$. 

To prove the reverse implication, assume that for some $f $ supported in $B$ with $f \in L^1 (B)$ we have $M(f) \in L_{\Psi} (B)$. Our task is to show that $f \in L \log \log L (B)$. In order to accomplish this, we shall make use of the fact that there exists a $\rho > 2$, depending only on $\| f \|_{L^1 (B)}$ and $B$, such that we also have $M (f) \in L_{\Psi} (\rho B)$ and moreover, for every $\alpha \geq e^e$,
\begin{equation}\label{wt_reverse}
| \{ x \in \rho B : M (f) (x) > c_1 \cdot \alpha \} | \geq \frac{c_2}{\alpha} \int_{B \cap \{ |f| > \alpha \}} |f(x)| dx, 
\end{equation}
where $c_1$, $c_2$ are positive constants that can be taken to be independent of $f$ and $\alpha$. Indeed, arguing as in the proof of \cite[Lemma 1]{Stein_LlogL}, note that for every $r > 2$ one has
\begin{equation}\label{away}
 M (f) (x) \lesssim \frac{1}{(r - 1)^d |B|} \cdot \| f \|_{L^1 (B)}  \quad \mathrm{for} \ \mathrm{all} \ x \in \mathbb{R}^d \setminus r B.
\end{equation}
Hence, if we choose $\rho >2 $ to be large enough, then $ M (f) (x) < e^e \leq \alpha$ for all $x \in \mathbb{R}^d \setminus \rho B$ and so, \eqref{wt_reverse}  follows from \cite[Inequality (6)]{Stein_LlogL}.

Furthermore, one can check that $M(f) \in L_{\Psi} (\rho B)$. Indeed, if we write $B = B (x_0, r_0)$ then, as in \cite{Stein_LlogL}, it follows from the definition of $M$ and the fact that $\mathrm{supp}(f) \subseteq B$ that there exists a constant $c_0>0$, depending only on the dimension, such that for every $ x \in 2B \setminus B$ one has 
\begin{equation}\label{pw_ineq_M}
M (f) (x ) \leq c_0 \cdot M (f) \left( x_0 + r_0^2 \cdot \frac{ x-x_0 } { | x - x_0 |^2} \right)
\end{equation}
and so, $M (f) \in L_{\Psi} (2 B)$. To show that \eqref{pw_ineq_M} implies that $M(f) \in L_{\Psi} (B)$, observe first that the function $\Psi_0 (s)=s/\log (e+s)$ is increasing on $[0,+\infty)$, and for all $t\geq 1$ and all $s>0$, 
\[
    1\geq \frac{\log(e+s)}{\log(e+ts)}=\frac{\log(e+s)}{\log(e/t+s)+\log t}\geq 
\frac{\log(e+s)}{\log(e+s)+\log t}\geq 
\frac{1}{1+\log t},
\]
so $\Psi_0$ satisfies
\begin{equation}\label{eq:doubling}
        t(1+\log t)^{-1} \Psi_0 (s)\leq \Psi_0 (st)\leq t\Psi_0 (s),
\end{equation} 
which implies that for all $c>0$ and all $s>0$ 
\[
    \Psi_0 (cs)\sim_c \Psi_0 (s).
\]

Observe that 
a change to polar coordinates, followed by another a change of variables and elementary estimates yield
\[
\begin{split}
        \int_{2B\setminus B} \Psi_0 (Mf(x))\dd x 
&\lesssim \int_{r_0}^{2r_0} s^{d-1}\int_{S^{d-1}} \Psi_0 (Mf(x_0+r_0^2 \theta/s))\dd \sigma(\theta)\dd s\\
&\sim r_0^{-1}\int_{\frac{1}{2}}^1  t^{-1-d}\int_{S^{d-1}} \Psi_0 (Mf(x_0+r_0t \theta))\dd \sigma(\theta)\dd t\\
&\sim \int_{\frac{1}{2}}^1  t^{d-1}\int_{S^{d-1}} \Psi_0 (Mf(x_0+r_0t \theta))\dd \sigma(\theta)\dd t \\
& \lesssim \int_B \Psi(x, Mf(x))\dd x.
\end{split}
\]

Moreover, we deduce from \eqref{away} that $M(f)$ belongs to $ L_{\Psi} (\rho B \setminus 2B)$ and it thus follows that $M (f) \in L_{\Psi} (\rho B)$, as desired. 

Next, note that by the same reasoning as in the proof of sufficiency and by Fubini's theorem,
\begin{align*}
\int_{\rho B} \Psi (x, M (f) (x)) dx &\gtrsim \int_{\rho B \cap  \{ M(f) > \max \{e^e, |x_0|+ r_0 \}  \} } \frac{M(f)(x)}{\log (M (f)(x))} dx \\
& \gtrsim \int_{\rho B \cap\{ M(f) > \max \{e^e, |x_0|+ r_0 \} \} } \Bigg( \int_{e^e}^{M(f)(x)} \frac{1}{ \log \alpha} d \alpha \Bigg) dx \\
& \gtrsim \int_{\max \{e^e, |x_0|+ r_0 \}}^{\infty} \frac{1}{\log \alpha} \cdot | \{x \in \rho B : M(f)(x) > c_2 \cdot \alpha \} | d \alpha  .
\end{align*}
By using \eqref{wt_reverse}, we now get
\begin{align*}
\infty  > \int_{\rho B} \Psi (x, M(f)(x)) dx & \gtrsim \int_B |f(x)| \cdot \Bigg( \int_{\max \{e^e, |x_0|+ r_0 \}}^{|f(x)|} \frac{1}{\alpha \log \alpha} d \alpha \Bigg) dx \\
&\gtrsim 1 + \int_B |f(x)| \log^+ \log^+ |f(x)| dx 
\end{align*}
and this completes the proof of Theorem \ref{Stein-type_lemma}.
\end{proof}

\begin{rmk} Let $B_0$ denote the closed unit ball in $\mathbb{R}^d$. Given a small $\delta \in (0,e^{-e})$, if, as on pp. 58--59 in \cite{Fefferman}, one considers $f := \delta^{-d} \chi_{\{ |x| < \delta \} }$ then $M (f) (x) \sim |x|^{-d}$ for all $|x| > 2 \delta$ and so,
\begin{equation}\label{equivalence_ex} 
\int_{B_0} |f(x)| \log^+ \log^+ |f (x)| dx \sim \log (\log (\delta^{-1}) ) \sim \int_{B_0} \Psi (x, M(f) (x)) dx .
\end{equation}
This shows that given $L_{\Psi} (B_0)$, the space $L \log \log L (B_0)$ in the statement of Theorem \ref{Stein-type_lemma} is best possible in general, in terms of size. 

Indeed, the left-hand side of \eqref{equivalence_ex} follows by direct calculation. On the other hand, using \eqref{eq:equivalence}, \eqref{eq:doubling}, a change to polar coordinates, and further change of variables yield
\[
   \begin{split}
       \int_{B_0} \Psi (x, M(f) (x)) dx &\sim 1+\int_{2\delta}^{1} \frac{1}{\log(e+s^{-d})}\frac{\dd s}{s}\\
&\sim 1+\int_{1}^{(2\delta)^{-1}} \frac{1}{\log(e+u^d)}\frac{\dd u}{u}\sim 1+\int_{e}^{(2\delta)^{-1}} \frac{1}{\log(u)}\frac{\dd u}{u},
   \end{split}
\]
from where the right-hand side of \eqref{equivalence_ex} follows.
\end{rmk}

\subsection{Further generalizations} Assume that $\Psi : \mathbb{R}^d \times [0, \infty)$ is a non-negative function satisfying the following properties:
\begin{enumerate}
\item For every $x \in \mathbb{R}^d$ fixed, $\Psi (x,t) = \Psi_x (t)$ is \emph{Orlicz} in $t \in [0,\infty)$, namely $\Psi_x (0) = 0$, $\Psi_x$ is increasing on $[0,\infty)$ with $\Psi_x (t) > 0$ for all $t >0$ and $\Psi_x (t) \rightarrow \infty$ as $t \rightarrow \infty$. 

Moreover, assume that there exists an absolute constant $C_0 > 0$ such that $\Psi_x (2t) \leq C_0 \Psi_x (t) $ for all $x \in \mathbb{R}^d$ and every $t \in [0, \infty) $. 
\item If $K$ is a compact set in $\mathbb{R}^d$, then there exist $x_1, x_2 \in K$ and a constant $C_K > 0$ such that
$$ C_K^{-1} < \Psi (x_1, t) \leq \Psi (x,t) \leq \Psi (x_2,t)  < C_K   $$
for every $x \in K$ and for all $t >0$. 
\item If we write $\Psi (x,t) = \Psi_x (t) = \int_0^t \psi_x (s) ds$, then for every $\alpha_0$, $\beta_0$ with $0 < \alpha_0 < \beta_0 $ one has
$$ \int_{\alpha_0}^{\beta_0} \frac{\psi_x (s)}{s} ds < \infty$$
for every $x \in \mathbb{R}^d$. 
\end{enumerate}

By carefully examining the proof of Theorem \ref{Stein-type_lemma}, one obtains the following result.

\begin{theo} \label{generalstein}

Let $\Psi (x,t) = \int_0^t \psi_x (s) ds$, $(x,t) \in \mathbb{R}^d \times [0, \infty)$, be as above.

Fix a closed ball $B$ with $B \subsetneq \mathbb{R}^d$ and let $f$ be such that $\mathrm{supp} (f) \subseteq B$.  Then, $M (f ) \in L_{\Psi} (B)$ if, and only if, 
$$ \int_{\{ | f| > \alpha_0 \} } | f(x) | \cdot \Bigg( \int_{\alpha_0}^{|f(x)|} \frac{\psi_x (s)}{s} ds \Bigg) dx < \infty $$
for every $\alpha_0 > 0$.
\end{theo}

Theorem \ref{generalstein} applies to certain Orlicz spaces considered in connection with convergence of Fourier series, see e.g. \cite{Antonov, Sjolin}, and the recent paper by V. Lie \cite{Lie}; we give some sample applications in Subsection \ref{applications_periodic}.

\section{Proof of the Zygmund-type Theorem for $H^{\log} (\mathbb{R}^d)$}
We begin with the following elementary lemmas.

\begin{lemma}\label{decreasing} Consider the function $g : [0, \infty )^2 \rightarrow [0, \infty)$ given by
$$ g (s,t) : = \frac{1}{ \log (e + t) + \log (e + s)}, \quad (s,t) \in [0, \infty )^2. $$
Then one has
$$ \Psi (x,t) \leq \int_0^t g(|x| , \tau) d \tau \leq 2 \Psi (x,t) $$ 
for all $(x,t) \in \mathbb{R}^d \times [0, \infty)$.
\end{lemma}
\begin{proof}
	The function $t\mapsto g(s,t)=\frac{1}{\log \brkt{(e+t)(e+s)}}$ is decreasing, so clearly 
	\[
	\int_0^t g(|x|,s)\dd s\geq tg(|x|,t)=\Psi(x,t).
	\] 
	We now address the upper bound. A calculation yields that
	\[
	\partial_t (t^\epsilon g(|x|,t))=\frac{t^{\epsilon}}{\log (e+t)+\log(1+\abs{x})}\brkt{\frac{\epsilon}{t}-\frac{1}{(e+t)(\log(e+t)+\log(e+\abs{x})}},
	\]
	and we observe that the term within the parenthesis is positive if, and only if,
	\[
	\frac{\epsilon}{t}-\frac{1}{(e+t)(\log(e+t)+\log(e+\abs{x})}>0,
	\]
	which for $\epsilon=\frac{1}{2}$ is equivalent to the inequality
	\[
	(e+t)(\log(e+t)+\log(e+\abs{x}))>2t.
	\]
	But clearly 
	\[
	(e+t)(\log(e+t)+\log(e+\abs{x}))\geq 2(e+t)>2t.
	\]
	Thus $s\mapsto s^\epsilon g(|x|,s)$ is increasing for $\epsilon =1/2$, which implies that 
	\[
	\int_0^t g(|x|,s)\dd s=\int_0^t s^{-\epsilon} s^{\epsilon}g(|x|,s)\dd s\leq \frac{1}{1-\epsilon}\Psi(x,t)=2\Psi(x,t)
	\]
and this completes the proof of the lemma.
\end{proof}

\begin{lemma}\label{translation}
Let $x_0 \in \mathbb{R}^d$ be fixed and for $u \in S(\mathbb{R}^d)$ define $\langle \tau_{x_0}f , u \rangle := \langle f, \tau_{-x_0}u \rangle$, where $\tau_{-x_0} u (x)  := u (x-x_0)$, $x \in \mathbb{R}^d$. 

Then $f\in H^{\log}(\mathbb{R}^d)$ if, and only if, $\tau_{x_0}f \in H^{\log}(\mathbb{R}^d)$.
\end{lemma}

\begin{proof} Note that it suffices to prove that for any  $x_0 \in \mathbb{R}^d$ and $f \in H^{\log}(\mathbb{R}^d)$ one also has that $\tau_{x_0}f \in H^{\log}(\mathbb{R}^d)$. 

Towards this aim, fix an $x_0 \in \mathbb{R}^d$ and an $f \in H^{\log} (\mathbb{R}^d)$. Observe that, by using a change of variables and the translation invariance of $M_{\phi}$, we may write 
$$ I: = \int_{\mathbb{R}^d} \frac{M_{\phi} (\tau_{x_0} f ) (x) }{ \log (e + |x|) + \log (e + M_{\phi} (\tau_{x_0} f) (x) ) } dx 
$$ as
$$ I = \int_{\mathbb{R}^d} \frac{M_{\phi} ( f ) (x ) }{ \log (e + |x-x_0|) + \log (e + M_{\phi} ( f) (x) ) } dx.  $$
To prove that $I < \infty$, we split
$$ I = I_1 + I_2 ,$$
where
$$ I_1 := \int_{ |x| > 4 |x_0| } \frac{M_{\phi} ( f ) (x ) }{ \log (e + |x-x_0|) + \log (e + M_{\phi} ( f) (x) ) } dx $$
and
$$ I_2 := \int_{ |x| \leq 4 |x_0| } \frac{M_{\phi} ( f ) (x ) }{ \log (e + |x-x_0|) + \log (e + M_{\phi} ( f) (x) ) } dx . $$
To show that $I_1 < \infty$, observe that for $|x| > 4 |x_0|$ one has
$$ \frac{4 |x-x_0| }{5}  < |x| < \frac{4 |x-x_0| } {3}  $$
and so,
\begin{align*}
 I_1 & \lesssim  \int_{ |x| > 4 |x_0| } \frac{M_{\phi} ( f ) (x ) }{ \log (e + |x|) + \log (e + M_{\phi} ( f) (x) ) } dx \\
& \leq \int_{ \mathbb{R}^d } \frac{M_{\phi} ( f ) (x ) }{ \log (e + |x|) + \log (e + M_{\phi} ( f) (x) ) } dx .
\end{align*}
Since $f \in H^{\log} (\mathbb{R}^d)$, the last integral is finite and we thus deduce that $I_1 < \infty$.  Next, to show that $I_2 < \infty$, we have 
\begin{align*}
I_2 & \leq \int_{ |x| \leq 4 |x_0| } \frac{M_{\phi} ( f ) (x ) }{ 1 + \log (e + M_{\phi} ( f) (x) ) } dx \\
& \lesssim_{|x_0|} \int_{ |x| \leq 4 |x_0| } \frac{M_{\phi} ( f ) (x ) }{ \log (e + |x|) + \log (e + M_{\phi} ( f) (x) ) } dx \\
& \leq  \int_{ \mathbb{R}^d } \frac{M_{\phi} ( f ) (x ) }{ \log (e + |x|) + \log (e + M_{\phi} ( f) (x) ) } dx
\end{align*}
and so, $I_2 < \infty$, as $f \in H^{\log} (\mathbb{R}^d)$. Therefore, $I < \infty$ and it thus follows that $\tau_{x_0} f \in H^{\log} (\mathbb{R}^d )$. \end{proof}

To obtain the desired variant of Zygmund's theorem, we shall use the fact that functions in $H^{\log} (\mathbb{R}^d)$ have mean zero; see  Lemma 1.4 in \cite{BIJZ}. For the convenience of the reader, we present a detailed proof of this fact below.

\begin{lemma}[\cite{BIJZ}] If $f \in H^{\log} (\mathbb{R}^d)$ is a compactly supported integrable function, then $\int_{\mathbb{R}^d} f (y) dy = 0$. 
\end{lemma}

\begin{proof} Let $f$ be a given function in $ H^{\log} (\mathbb{R}^d)$ with compact support. In light of Lemma \ref{translation}, we may assume, without loss of generality,  that $f$ is supported in a closed ball $B_r$ centered at $0$ with radius  $r >0$, i.e. $\mathrm{supp} (f) \subseteq B_r := \{ x \in \mathbb{R}^d : |x| \leq r \}$. 

To prove the lemma, take an $x \in \mathbb{R}^d$ with $|x| > 2r$ and observe that, by the definition of $\phi_{\epsilon}$, we can take $\epsilon=4|x|$ to get
\[| f \ast \phi_{\epsilon} (x) |=\frac{1}{\epsilon^d}\left|\int_{B_r}f(y)\phi\left(\frac{x-y}{\epsilon}\right)dy\right| \gtrsim \frac{1}{|x|^d} \cdot \left| \int_{B_r} f(y) dy \right|\]
as we then have $\phi(\epsilon^{-1}(x-y))=c_0$ for $y\in B_r$. Therefore, for all $|x| > 2r$ and $\epsilon=4|x|$, we have
$$ M_{\phi} (f) (x) \gtrsim \frac{1}{|x|^d} \cdot \Bigg| \int_{B_r} f(y) dy \Bigg| $$
and so, we deduce from Lemma \ref{decreasing} that
$$ \Psi (x, M_{\phi}(f) (x)) \gtrsim \frac{1}{|x|^d \log (e + |x|)} \cdot \Bigg| \int_{B_r} f (y) dy \Bigg| $$
for $|x|$ large enough.

Hence, if $\int f(y)dy \neq 0$, then the function $\Psi (x, M_{\phi} (f)(x)) $ does not belong to $ L^1 (\mathbb{R}^d)$, which is a contradiction.
\end{proof}

We are now ready to prove Theorem \ref{LlogL}.



\begin{proof}[Proof of Theorem \ref{LlogL}] Let $B$ denote the unit closed ball in $\mathbb{R}^d$. Fix a function $f$ with $\mathrm{supp} (f) \subseteq B $, $\int_B f (y) dy = 0$ and $f \in L \log \log L (B )$. First of all, observe that 
$$ M_{\phi} (f) (x) \lesssim M(f)(x) \quad \mathrm{for}\ \mathrm{all}\ x \in \mathbb{R}^d, $$
where $M(f)$ denotes the Hardy-Littlewood maximal function of $f$; see e.g. Theorem 2 on pp. 62--63 in \cite{Singular}. We thus deduce from Lemma \ref{decreasing} that
$$  \Psi (x, M_{\phi} (f)(x)) \lesssim \Psi (x, M(f)(x))  \quad \mathrm{for}\ \mathrm{all}\ x \in \mathbb{R}^d $$
and hence, by using Theorem \ref{Stein-type_lemma}, we obtain
\begin{equation}\label{local_bound}
 \int_{2 B } \Psi (x, M_{\phi}(f)(x)) dx \lesssim 1 + \int_{2 B } |f(x)| \log^+ \log^+ |f(x)| dx,
\end{equation}
where $2 B  : = \{x \in \mathbb{R}^d : |x| \leq 2\} $.

To estimate the integral of $\Psi (x, M_{\phi} (f) (x))$ for $x \in \mathbb{R}^d \setminus 2 B $, we shall make use of the cancellation of $f$. To be more specific, observe that if $|x|>2$ then for every $\epsilon < |x|/2$, one has that
$$f \ast \phi_{\epsilon}(x) = \frac{1}{\epsilon^d} \int_{B } f(y) \phi \Big( \frac{x-y}{\epsilon} \Big) dy = 0$$
since $|x-y|/\epsilon > 1$ whenever $y \in B $. Therefore, we may restrict ourselves to $\epsilon \geq |x|/2$ when $|x|>2$. Hence, for $\epsilon \geq |x|/2$, by exploiting the cancellation of $f$ and using a Lipschitz estimate on $\phi_{\epsilon}$, we obtain
\begin{align*}
|f \ast \phi_{\epsilon} (x)| &= \frac{1}{\epsilon^d} \Bigg| \int_B f(y)  \phi \Big( \frac{x-y}{\epsilon} \Big) dy \Bigg| = \frac{1}{\epsilon^d} \Bigg| \int_B f(y) \Big[  \phi \Big( \frac{x-y}{\epsilon} \Big) - \phi \Big( \frac{x}{\epsilon} \Big) \Big] dy \Bigg| \\
&\lesssim_{\phi} \frac{1}{\epsilon^{d+1}} \int_B |y \cdot f(y)| dy \lesssim \frac{1}{|x|^{d+1} } \Bigg[ 1 + \int_B |f(y)| \log^+ \log^+ |f(y)| dy \Bigg].
\end{align*}
We thus deduce that, for every $x \in \mathbb{R}^d \setminus 2 B $,
$$ |M_{\phi} (f) (x)| \lesssim \frac{1}{|x|^{d+1}} \Bigg[1 + \int_B |f(y)| \log^+ \log^+ |f(y)| dy \Bigg] $$
and so, 
\begin{align*}
& \int_{\mathbb{R}^d \setminus 2 B } \Psi (x, M_{\phi} (x)) dx  \\
&  \lesssim \Bigg[ 1 + \int_B |f(y)| \log^+ \log^+ |f(y)| dy \Bigg] \cdot  \int_{\mathbb{R}^d \setminus 2 B } \frac{1} {|x|^{d+1 }\log (e +|x|)} dx \\
& \lesssim 1 + \int_B |f(y)| \log^+ \log^+ |f(y)| dy,
\end{align*}
as desired. Therefore, Theorem \ref{LlogL} is now established by using the last estimate combined with \eqref{local_bound}.
\end{proof}

\subsection{A partial converse}\label{periodic}

As in the classical setting of the real Hardy space $H^1$, see \cite{Stein_LlogL}, Theorem \ref{LlogL} has a partial converse. To be more precise, if a function $f$ is positive on an open set $U$ and $f$ belongs to $ H^{\log} (\mathbb{R}^d)$, then the function $f \in L \log \log L (K)$ for every compact set $K\subset U$. 

Indeed, to see this, note that if $f$ is as above then
$$ M_{\phi} (f) (x) \gtrsim M (f \cdot \eta_K ) (x) \quad \mathrm{for}\ \mathrm{all} \ x \in K,$$
where $\eta_K$ is an appropriate Schwartz function with $\eta_K \sim 1$ on $K$; see e.g. Section 5.3 in Chapter III in \cite{Big_Stein}. Hence, by using Lemma \ref{decreasing} and Theorem \ref{Stein-type_lemma}, we get
\begin{align*}
 \int_{\mathbb{R}^d}  \Psi (x, M_{\phi} (f) (x) ) dx \geq  \int_K  \Psi (x, M_{\phi} (f) (x) ) dx &\gtrsim \int_K \Psi (x, M (\eta_K \cdot f) (x) ) dx \\
&  \gtrsim  1 +   \int_K |f(x)| \log^+ \log^+ |f(x)| dx. 
\end{align*}


\section{Variants in the periodic setting} Following \cite{BIJZ}, define $H^{\log} (\mathbb{D})$ to be the space of all holomorphic functions $F$ on the unit disk $\mathbb{D}$ of $\mathbb{C}$ such that
$$ \sup_{0 < r < 1} \int_0^{2\pi} \frac{|F (r e^{i \theta })|}{ \log (e + |F (r e^{i \theta})|) } d \theta < \infty .$$

For $0<q \leq \infty$, let $H^q (\mathbb{D})$ denote the classical Hardy space on $\mathbb{D}$ consisting of analytic functions $F$ having
\[\sup_{0<r<1}\int_{0}^{2\pi}|F(re^{i\theta})|^pd\theta<\infty;\]
see for instance \cite{Duren}. Then 
$$ H^1 (\mathbb{D}) \subsetneq H^{\log} (\mathbb{D}) \subsetneq H^p (\mathbb{D}) \quad \mathrm{for} \ \mathrm{all} \ 0<p<1 .$$ Hence, if $F \in H^{\log} (\mathbb{D})$ then $F$ 
has a non-tangential limit $F^*$ at almost every point of $\mathbb{T}=\partial \mathbb{D}$, and this non-tangential limit lies in $L^p(\mathbb{T})$ for $0 < p <1$. See \cite[Theorem 2.2]{Duren} for details.  Moreover, by using \cite[Proposition 8.2]{BIJZ}, one may identify $ H^{\log} (\mathbb{D})$ with the space of all measurable functions $f$ on the torus such that
$$\int_0^{2\pi} \Psi_0 \Big(  \sup_{0 < r < 1}  |  P_r \ast f  ( \theta)| \Big) d \theta < \infty , $$
where $\Psi_0 (t) : = t \cdot [\log (e +t)]^{-1}$ ($t \geq 0$) and for $0<r<1$, $\theta \in [0,2\pi)$, 
\[P_r (\theta) := \frac{1-r^2}{1 -2r \cos (\theta) + r^2}\]
denotes the Poisson kernel in the unit disk.  

There is a periodic version of Theorem \ref{Stein-type_lemma}, namely $M(f) \in L_{\Psi_0} (\mathbb{T})$ if, and only if, $f \in L \log \log L (\mathbb{T})$. Combining this with Lemma \ref{decreasing}, one obtains the following result.

\begin{proposition}\label{inclusion} One has the inclusion
$$ L \log \log L (\mathbb{T}) \subseteq  H^{\log} (\mathbb{T}) .$$
\end{proposition}

Moreover, arguing as in the previous section and using the necessity in Theorem \ref{Stein-type_lemma} as well as Proposition \ref{inclusion} and Lemma \ref{decreasing}, one can show that if $f \in H^{\log} (\mathbb{T})$ and $f$ is non-negative, then $f \in L \log \log L (\mathbb{T})$. 

\begin{proposition}\label{identification} 
One has $$ \{ f \in L \log \log L (\mathbb{T}) : f\geq 0 \  \mathrm{a.e.}\ \mathrm{on}\ \mathbb{T}\} = \{ f\in  H^{\log} (\mathbb{T}) : f \geq 0\  \mathrm{a.e.}\ \mathrm{on}\ \mathbb{T} \} .$$
\end{proposition}

\begin{proof} Note that Proposition \ref{inclusion} implies that
\begin{equation}\label{incl_1}
\{ f \in L \log \log L (\mathbb{T}) : f\geq 0 \  \mathrm{a.e.}\ \mathrm{on}\ \mathbb{T}\} \subseteq \{ f \in  H^{\log} (\mathbb{T})   : f \geq 0\  \mathrm{a.e.}\ \mathrm{on}\ \mathbb{T} \} .
\end{equation}

To prove the reverse inclusion, take a non-negative function $f \in H^{\log} (\mathbb{T})$ and notice that it follows from the work of Stein \cite{Stein_LlogL} that
\begin{equation}\label{wt_reverse_periodic} 
 |\{ \theta \in \mathbb{T} : M (f) ( \theta ) > c_1\alpha \} | \geq\frac{c_2}{\alpha} \int_{\{ |f| > \alpha \}} |f( \theta )| d \theta, 
\end{equation}
where $c_1,c_2 > 0$ are absolute constants. Hence, by arguing as in the proof of Theorem \ref{Stein-type_lemma}, it follows from \eqref{wt_reverse_periodic} (noting that the periodic case is easier as one does not need to consider the contribution away from the support of $f$) that
\begin{equation}\label{reverse_periodic}
 \int_{\mathbb{T}} {\Psi}_0 (M (f)) (\theta) d \theta \gtrsim 1 + \int_{\mathbb{T}} |f(x)| \log^+ \log^+ |f(\theta)| d \theta . 
\end{equation}
Since $f \geq 0$ a.e. on $\mathbb{T}$, as in the Euclidean case, one has 
\begin{equation}\label{pw_per}
 \sup_{0< r < 1} | P_r \ast f (\theta) | \gtrsim M (f) (\theta) \quad \mathrm{for}\  \mathrm{a.e.}\ \theta \in \mathbb{T}. 
\end{equation}
Hence, by using \eqref{reverse_periodic}, \eqref{pw_per} and Lemma \ref{decreasing}, we deduce that $f \in L \log \log L (\mathbb{T})$ and so,
\begin{equation}\label{incl_2}
\{ f \in  H^{\log} (\mathbb{T})  : f \geq 0\  \mathrm{a.e.}\ \mathrm{on}\ \mathbb{T} \} \subseteq
\{ f \in L \log \log L (\mathbb{T}) : f\geq 0 \  \mathrm{a.e.}\ \mathrm{on}\ \mathbb{T}\}.
\end{equation}
The desired fact is a consequence of \eqref{incl_1} and \eqref{incl_2}.
\end{proof}

\subsection{{Some further applications}}\label{applications_periodic}

We conclude with some applications of Theorem \ref{generalstein} in the periodic setting.
The function \[\Psi(x,t)=\Psi  (t)=t \log^+t\,\log^+\log^+t\] appearing in \cite{Sjolin} satisfies the hypotheses of Theorem \ref{generalstein}, and we now determine which space maps into $L_{\Psi}$ via the maximal function. With the associated $\psi$ defined as before, an integration by parts yields
\[\int \frac{\psi(s)}{s}ds=\frac{1}{2}(\log^+s)^2\log^+\log^+s+\log^+s\log^+\log^+s-\frac{1}{4}(\log^+s)^2.\]
This allows us to conclude that, for this choice of $\Psi$, 
\[M(f)\in L_{\Psi}(\mathbb{T})\quad \textrm {if, and only if,}\quad  f\in L\log^2L\log\log L(\mathbb{T}).\]
Turning to the space $L\log\log L \log\log\log\log L$ appearing in Lie's paper \cite{Lie}, we can check where the maximal operator maps this space. Performing the appropriate computations, we obtain that
\[\int_{\mathbb{T}}\frac{M(f)}{\log(M(f)+e)}\log^+\log^+\log^+\log^+M(f)dx<\infty\]
if, and only if,
\[f\in L\log \log L\log \log \log \log L (\T).\]

Roughly speaking, the contents of Theorem \ref{generalstein} and the computations presented above can be summarized as follows. Let $\Phi_0$ be a given Orlicz function, namely $\Phi_0 : [0 , \infty) \rightarrow [0, \infty) $ is an increasing function with $\Phi_0 (0) = 0$ and $\Phi_0 (t) \rightarrow \infty$ as $t \rightarrow \infty$. Suppose that one can find non-negative, increasing functions $M,S$ with
$$ \Phi_0 (t) = M (t) \cdot S(t) \quad (t>0)$$
and such that, for $0 < \alpha < t$, one can easily compute
$$ F_{\alpha} (t) := \int_{\alpha}^t \frac{M'(s)}{s} d s $$
in closed form and, moreover, that there exists an $\alpha_0 >0$ with the property that for every $\alpha \geq \alpha_0$ one has
$$ F_{\alpha} (t) \cdot S(t) \gtrsim  \int_{\alpha}^t \Big( \frac{M (s)}{s} + F (s) \Big) \cdot S' (s) d s  \quad \mathrm{for}\ \mathrm{all} \ t \geq \alpha . $$
Then, by arguing as in Section \ref{steinproof}, one deduces the ``concrete'' relation
$$ f \in L_{\Phi_0} (\mathbb{T}) \quad \mathrm{if}, \ \mathrm{and}\ \mathrm{only}\ \mathrm{if}, \quad M(f) \in L_{F_{\alpha} \cdot S} (\mathbb{T}),$$
for any $\alpha \geq \alpha_0$.

\subsection*{Acknowledgments}
AS extends his thanks to Kelly Bickel and the rest of the Department of Mathematics at Bucknell University (Lewisburg, PA) for hospitality during a visit where part of this work was carried out.


\end{document}